\documentclass[11pt]{article} 
\usepackage{multicol} 
\usepackage{newlfont}
\usepackage{amsmath,amssymb}

\def\zit{\mathbb{Z}}   
\def\nit{\mathbb{N}}

\def\ppit{\mathbb{P}} 
 
\def\cit{\mathbb{C}}

\def\yit{\mathbb{Y}}

\newcommand{\pf}{{\em Proof.~}}
\newcommand{\qed}{\hfill~~\mbox{$\Box$}}

\newenvironment{proof}{\smallskip \noindent \pf}{\qed \bigskip}

\newtheorem{theorem}{Theorem}[subsection]
\newtheorem{proposition}[theorem]{Proposition}

\newtheorem{lemma}[theorem]{Lemma}
\newtheorem{corollary}[theorem]{Corollary}

\newtheorem{remark}[theorem]{Remark}
\newtheorem{example}[theorem]{Example}
\newtheorem{conjecture}[theorem]{Conjecture}

\DeclareMathOperator{\Rang}{Rank}
\DeclareMathOperator{\coker}{coker}
\DeclareMathOperator{\diag}{diag}
\DeclareMathOperator{\codim}{codim}
\DeclareMathOperator{\im}{im}
\DeclareMathOperator{\grad}{grad}
\DeclareMathOperator{\Irr}{Irr}
\DeclareMathOperator{\Reg}{Reg}

\begin{document}

\title{\bf Gauss-Manin systems of wild regular functions: Givental-Hori-Vafa models of smooth hypersurfaces in weighted projective spaces as an example}
\author{\sc Antoine Douai \thanks{Partially supported by the grant ANR-13-IS01-0001-01 of the Agence nationale de la recherche.
Mathematics Subject Classification 32S40 14J33 34M35}\\
Laboratoire J.A Dieudonn\'e, UMR CNRS 7351, \\
Universit\'e de Nice, Parc Valrose, F-06108 Nice Cedex 2, France \\
Email address: Antoine.DOUAI@unice.fr}

\maketitle

\begin{abstract}
We study Gauss-Manin systems of non tame Laurent polynomial functions. We focuse on Givental-Hori-Vafa 
models, which are the expected mirror partners of the small quantum cohomology of smooth hypersurfaces in weighted projectives spaces.
\end{abstract}

\section{Introduction}

An interesting feature of mirror symmetry is that it suggests the study of new, and sometimes unexpected, phenomena, on the A-side (quantum cohomology) as well on the B-side 
(singularities of regular functions). 
From this point of view, the case of (the contribution of the ambient part to) the small quantum cohomology of smooth hypersurfaces 
in weighted projective spaces is particularly significant and leads to 
the study of a remarkable class of regular functions on the torus, the Givental-Hori-Vafa models, see \cite{Giv0},
\cite{HV}, \cite{P} and section \ref{sec:ModelesHV}. 
The main observation is that, unlike the usual absolute situation, considered for instance in \cite{Do1}, \cite{DoMa}, \cite{DoSa1}, \cite{DoSa2}, \cite{S},  such functions are not tame 
and may have some singular points at infinity
(in words, $f$ is tame if the set outside which $f$ 
is a locally trivial fibration is made from critical values of $f$ and that these critical values belong to this set only because of the 
critical points at 
finite distance, see section \ref{sec:CyclesEvInfini}): therefore a geometric situation requires the study of wild functions, and this can be done 
for instance using the previous works of Dimca-Saito\cite{DS} and Sabbah \cite{S}.
One aim of these notes is to enlighten this interaction between singularities of functions, Gauss-Manin systems, 
smooth hypersurfaces in weighted projective spaces, quantum cohomology and to connect rather classical results in various domains. 
For instance, it's worth to note that an arithmetical condition that ensures the smoothness of a hypersurface in a weighted projective space 
gives also a
number of vanishing cycles at infinity for the expected mirror partner, see sections \ref{sec:Conjectures} and \ref{sec:ProbBir}.

We proceed as follows: in a first part, we focuse on Gauss-Manin systems of regular functions, emphasizing their relations with 
 singular points (including at infinity), see sections 
\ref{sec:CyclesEvInfini} and \ref{sec:F(GM)}. 
It happens that, and this is a major difference with the tame case, the Brieskorn module of a Givental-Hori-Vafa model $f$ is not of finite 
type, and this is due to some vanishing cycles produced by singular points at infinity, see section \ref{sec:F(GM)}.
We discuss an explicit characterization of these singular points at infinity. 
Notice that the situation is slightly different from the classical polynomial case considered in \cite{Br}, 
\cite{Paru} for instance: as our Givental-Hori-Vafa model are Laurent polynomials,
we have also to take into account the singular points on the polar locus of $f$ at finite distance. Fortunately, the results in \cite{ST}, \cite{T} fit in very well with this situation.

In a second part, we are interested in the following formulation of mirror symmetry:
above the small quatum cohomology of a degree $d$ hypersurface in a projective space (and we consider here only the contribution of the 
ambient space to the small quantum cohomology, see \cite{CK}, \cite{Giv}, \cite{MM} and section \ref{sec:PCQ}) and above a Givental-Hori-Vafa model on the $B$-side, we make grow a quantum differential system 
in the sense of \cite{Do1}, \cite{DoMa}. Two models will be mirror partners if their respective quantum differential systems are isomorphic. 
On the $B$-side, the expected quantum differential system
can be constructed solving a Birkhoff problem for the Givental-Hori-Vafa model alluded to, see
section \ref{sec:BirHV}.
In the tame case, this bundle is provided by the Brieskorn module,
which is in this situation free of finite rank as it
follows from the tameness assumption. 
But, and as previously noticed, this will be certainly not the case for Givental-Hori-Vafa models of smooth hypersurfaces, and we have to imagine something else.
We give a general result in this way for quadrics in $\ppit^n$, and this was, after \cite{GS}, one of the triggering factors of this paper.
Precisely,
let $G$ be the (localized Fourier transform of the)  Gauss-Manin system of the  
Givental-Hori-Vafa model $f$ of a smooth quadric in $\ppit^n$, see sections \ref{sec:F(GM)} and \ref{sec:ModelesHV}. It turns out that
$G$ is a connection 
and we show in section \ref{sec:QuadPnB}  the following result:
\begin{theorem}\label{theo:Principal}
We have a direct sum decomposition 
\begin{equation}\nonumber
G=H\oplus H^{\circ}
\end{equation}
of free modules where
$H$ is free of rank $n$ and is equipped with a connection making it isomorphic to the 
differential system associated with the small quantum cohomology of quadrics in $\ppit^n$.
\end{theorem}

\noindent In the previous situation, the rank of $G$ is equal to the global Milnor number of $f$ (the number of critical points  
of $f$ at finite distance, counted with multiplicities, that is $n-1$) plus a number of vanishing cycles at infinity: we expect that the latter is equal to one, in other words that the rank of $G$ is precisely equal to $n$ (and $H^{\circ} =0$), see conjecture \ref{conj:SingInftyHV} and corollary \ref{coro:RangG}. 
This is what happens for instance for $n=3$ et $n=4$, see example \ref{sec:Q2P3} 
(the case $n=4$ is also considered in \cite{GS}, using a different strategy).\\

These notes are organized as follows: 
in section 2 we discuss about the topology of regular functions and we study their Gauss-Manin systems in section 3. 
In section 4 we gather the results about hypersurfaces in weighted projective spaces that we need in order to define Givental-Hori-Vafa models in section 5.  
Their relationship with mirror symmetry is emphasized in section 6. As an application, we study the case of the quadrics in section 6.2.

 \section{Topology of regular functions}
\label{sec:CyclesEvInfini}

We collect in this section the general results about topology of regular functions that we will need. Our references are  \cite{D}, \cite{DS} and \cite{S}. The exposition is borrowed from the old preprint \cite{Do}.

\subsection{Isolated singularities including at infinity}
\label{sec:SingIsolInfty}
Let $U$ be an affine manifold of dimension $n\geq 2$, $S=\cit$ and $f:U\rightarrow S$ be a regular function. 
We will say that $f$  has {\em isolated singularities including at infinity} if there exists a  compactification
$$\overline{f}:X\rightarrow S$$
of $f$, $X$ is quasi-projective and $\overline{f}$ is proper, such that the support $\Sigma_{s}$ of 
$\varphi_{\overline{f}-s}Rj_{*}\cit_{U}$ 
is at most a finite number of points. Here, 
$j:U\rightarrow X$ denotes the inclusion and
$\varphi$ denotes the vanishing cycles functor \cite[Chapter 4]{D}.
If it happens to be the case, $f$ has at most isolated critical points on $U$ \cite[Theorem 6.3.17]{D}. If moreover $\Sigma_{s}\subset U$ for all $s\in S$, $f$ is said to be  
{\em cohomologically tame} \cite{S}.

Let us assume that $f$ has isolated singularities including at infinity. Since $^{p}\varphi :=\varphi [1]$ preserves peverse sheaves, 
${\cal E}_{s}:=\ ^{p}\varphi_{\overline{f}-s}Rj_{*}\cit_{U}[n]$ is a perverse sheaf with support in $\Sigma_{s}$ and thus
${\cal H}^{i}({\cal E}_{s})=0$ for $i\neq 0$,
because $\Sigma_{s}$ has ponctual support, see \cite[Example 5.2.23]{D}. For $x\in \Sigma_{s}$, the fibre $E_{x}:={\cal H}^{0}({\cal E}_{s})_{x}$ is a finite dimensional vector space:
\begin{itemize}
\item if $x\in U$ we have 
\begin{equation}\label{eq:MilnorMu}
\dim E_{x}=\mu_{x}\ \mbox{and}\ \sum_{x\in U}\dim E_{x}=\mu 
\end{equation}
$\mu_{x}$ denoting the Milnor number of $f$ at $x$ 
and $\mu $ the global  Milnor number of $f$, see \cite[proposition 6.2.19]{D}; we will also write $\mu_{s}=\sum_{x\in \Sigma_s \cap U}\dim E_{x}$,
\item if $x\in\Sigma_{s}\cap (X-U)$ we define 
\begin{equation}\label{eq:NotNu}
\nu_{x,s}:=\dim E_{x},\ \nu_{s}:=\sum_{x\in\Sigma_{s}\cap (X-U)} \nu_{x,s}\ \mbox{and}\ 
\nu :=\sum_{x\in X-U}\nu_{x,s}. 
\end{equation}
\end{itemize}
\noindent 
The point $x\in\Sigma_{s}\cap (X-U)$
is a {\em singular point of $f$ at infinity} if $\nu_{x,s}>0$.

Let $^{p}{\cal H}^{i}$ be the perverse cohomology functor: one has, see for instance \cite[Theorem 5.3.3]{D},
$$DR({\cal M}^{(i)})=\ ^{p}{\cal H}^{i}(Rf_{*}\cit_{U}[n])$$
where the ${\cal M}^{(i)}$'s are the cohomology groups of the direct image $f_{+}{\cal O}_U$ (the degrees are shifted $n=\dim U$).  
If $f$ has isolated singularities including at infinity, the perverse sheaves $^p {\cal H}^{i}(Rf_{*}\cit_{U}[n])$
are locally constant on $S$ for $i\neq 0$ because $\varphi_{t-s}(^p {\cal H}^{i}(Rf_{*}\cit_{U}[n]))=0$ if $i\neq 0$ for all $s\in\cit$
\cite[3.1.1]{DS} and \cite[Exercise 4.2.13]{D}. It follows that 
${\cal H}^{0}(^p {\cal H}^{i}(Rf_{*}\cit_{U}[n]))=0$ and that ${\cal H}^{-1}(^p {\cal H}^{i}(Rf_{*}\cit_{U}[n]))$
is a constant sheaf on $S$  for $i\neq 0$.
One has also, using the characterization of perverse sheaves in dimension $1$
\cite[Proposition 5.3.6]{D},
\begin{equation}\label{eq:SuiteExactePerv}
0\rightarrow {\cal H}^{0}(^p {\cal H}^{i}(Rf_{*}\cit_{U}))\rightarrow R^{i}f_{*}\cit_{U}\rightarrow {\cal H}^{-1}(^p {\cal H}^{i+1}(Rf_{*}\cit_{U}))\rightarrow 0
\end{equation}
and therefore
\begin{equation}\label{eqH-iP}
^p {\cal H}^{i}(Rf_{*}\cit_{U})=(R^{i-1}f_{*}\cit_{U})[1]
\end{equation}
for all $i<n$ because ${\cal H}^{0}(^p {\cal H}^{i}(Rf_{*}\cit_{U}))=0$ for $i<n$.
Notice also that
\begin{equation}\label{eq:H-1P}
^p {\cal H}^{n}(Rf_{*}\cit_{U})=(R^{n-1}f_{*}\cit_{U})[1]
\end{equation}
if $R^{n}f_{*}\cit_{U}=0$. Last, and by definition, 
$\dim_{\cit}\ ^{p}\varphi_{t_{s}}\ ^p {\cal H}^{n}(Rf_{*}\cit_{U})=\mu_s +\nu_s $.

\begin{proposition}\cite[Proposition 6.2.19]{D}, \cite{DS}
\label{prop:RangM}
 Let $f:U\rightarrow S$ be a regular function, with isolated singularities 
including at infinity. 
\begin{enumerate}
\item One has
\begin{equation}\label{eq:rangRn-1}
m=\mu +\nu +h^{n-1}(U)-h^{n}(U)
\end{equation}
where $m$ is the rank of $R^{n-1}f_{*}\cit_{U|V}$, $V=S-\Delta$ denoting the maximal open set in $S$ on which the constructible sheaf $R^{n-1}f_{*}\cit_{U}$ is a local system.
\item One has
\begin{equation}\label{eq:FibreSpeFibreGen}
\chi (f^{-1}(s'))-\chi(f^{-1}(s))=(-1)^{n-1}(\mu_{s}+\nu_{s})
\end{equation}
for all $s, s'\in S$ such that $s'\notin\Delta$. 
\end{enumerate}
\end{proposition}\qed

\noindent We will use the previous proposition it in order to compute the rank of the Fourier transform of the Gauss-Manin system of 
 some regular functions, see theorem \ref{theo:rangMG}.
 
\begin{remark}\label{rem:nuind}
 Formula (\ref{eq:FibreSpeFibreGen}) shows that the number of vanishing cycles at infinity $\nu$ defined by (\ref{eq:NotNu}) is precisely the one 
 defined  by Siersma and Tibar and denoted by $\lambda$ in
\cite[corollary 4.10]{ST}, \cite[paragraphe 3]{T}. It also shows that
the numbers $\nu_{s}$ do not depend on the choosen compactification of $f$. 
\end{remark}

There exists a finite set  
$B \subset \cit$ such that
$f:U- f^{-1}(B)\rightarrow \cit -B$
is a locally trivial fibration.
The smallest such set, denoted by $B(f)$, is called {\em the bifurcation set} of $f$  and its points are called the {\em atypical values}. 
A value which is not atypical is  {\em typical}.
In general  $B(f)=C(f)\cup B_{\infty}(f)$ where $C(f)$ is 
the set of critical values of $f$ and $B_{\infty}(f)$ is a contribution from singular points at infinity. 
One can be more precise if $f$ has isolated singularities including at infinity:

\begin{proposition}\cite[Theorem 4.12]{ST} Let $f$ be a (Laurent) polynomial 
 with isolated singularities including at infinity.
Then $a$ is typical if and only if $\nu_a =\mu_a =0$.
\end{proposition}


\subsection{Vanishing cycles at infinity with respect to the projective compactification by the graph}
\label{sec:CompStandard}

We apply the previous definitions to Laurent polynomials, using the standard compactification by the graph. 
We follow here \cite{ST} and \cite{T}.
Let $Y=\ppit^{n}$ and 
\begin{equation} \nonumber
F:Y\dashrightarrow\ppit^{1}
\end{equation}
be the rational function defined by $F(x)=(P(x):Q(x))$ where $P$ and $Q$ are two homogeneous polynomials of same degree.
Let
\begin{equation}\nonumber 
G=\{(x,t)\in (Y-A)\times \ppit^{1}\ | \ F(x)=t\}
\end{equation}
where $A=\{x\in Y | \ P(x)=Q(x)=0\}$ and
\begin{equation} \label{def:grapheY}
\yit =\{(x,(s:r))\in Y\times \ppit^{1}\ | \ rP(x)=sQ(x)\}
\end{equation}
be the closure of $G$ in $Y\times \ppit^{1}$. The singular locus $\yit_{sing}$ of $\yit$ is contained in $A$. By definition, $G$ is the graph of $F$ restricted to $Y-A$ and thus $G\simeq Y-A$: the inclusion 
$Y-A \hookrightarrow \yit$ defines the compactification 
\begin{align}\nonumber
\begin{array}{ccc}
Y-A & \hookrightarrow & \yit \\
    & \searrow        &\downarrow \pi \\
    &                   & \ppit^{1}
\end{array}
\end{align}
of $F$, $\pi$ denoting the projection on the second factor. 
With the notations 
of section \ref{sec:SingIsolInfty}, $X=\yit$ and $\pi =\overline{f}$.

Assume now that the hypersurface $\yit_{a}:=\pi^{-1}(a)$ has an {\em isolated} singularity at $(p,a)\in A\times \{a\}$
and denote by $\mu_{p,a}$ the corresponding Milnor number.
If $\yit_{sing}$ is a curve at $(p,a)$, it intersects  $\yit_{s}$,  
$s$ close to $a$, at points $p_i (s)$, $1\leq i\leq k$.
Let $\mu_{p_i (s), s}$ be the Milnor number of $\yit_{s}$ at $p_i (s)$.

\begin{proposition}\label{prop:nu}
Assume that $\yit_{a}$ has an isolated singularity at $(p,a)\in A\times \{a\}$. Then
$$\nu_{p,a}= \mu_{p,a}-\sum_{i=1}^{k}\mu_{p_i (s), s}.$$ 
\end{proposition}
\begin{proof}
Follows from remark \ref{rem:nuind} (1) and \cite[Theorem 5.1]{ST}. 
\end{proof}

\noindent This proposition is very explicit when $\yit_{sing}$ is a line $\{p\}\times \cit$:     
indeed, let $\mu_{p, gen}$
be the Milnor number of the hypersurface $\yit_{s}$ at $p$ for generic $s$.

\begin{corollary}\label{coro:nu}
Assume that $\yit_{sing}=\{p\}\times \cit$. Then  
$\nu_{p,a}= \mu_{p,a}-\mu_{p, gen}$.\qed
\end{corollary}

\begin{remark}
\label{rem:CompStandardLaurent}
We will apply the previous construction to Laurent polynomials 
$$f(x_1 ,\cdots ,x_n)=\frac{P(x_1 ,\cdots ,x_n)}{Q(x_1 ,\cdots ,x_n )}$$
 where
$P$ and $Q$ have no common factors, $Q$ is monomial and $\deg P\geq \deg Q$. 
The homogeneization of $f$ is
$$\frac{P(X_0 , X_1 ,\cdots , X_n )}{Q(X_0 , X_1 ,\cdots ,X_n )}:=\frac{X_0 ^{\deg P}P(X_1 /X_0 ,\cdots ,X_n /X_0 )}{X_0 ^{\deg P}Q(X_1 /X_0 ,\cdots ,X_n /X_0 )}$$
and we will write
$$F(X_0 ,X_1 ,\cdots , X_n ,t):= P(X_0 , X_1 ,\cdots , X_n )-tQ(X_0 , X_1 ,\cdots ,X_n )$$
for $t\in\cit$. 
\end{remark}

\begin{remark}
\label{rem:NonCrit}
By \cite[Theorem 1.3]{Paru}, a polynomial function $f$ is cohomologically tame for the standard projective compactification
by the graph if and only if $f$ satisfies Malgrange's condition
$$\exists \delta >0,\  |x||\partial f(x)|\geq\delta\ \mbox{for}\ |x|\ \mbox{large enough},$$
$\partial f(x)$ denoting the gradient of $f$ at $x$.
Let now $f$ be a  Laurent polynomial and define 
\begin{equation}
T_{\infty}(f)=\{c\in \cit |\ \exists\  (p_{n}),\
p_{n}\rightarrow p \in \ppit^n -U,\ p_{n}\grad f(p_{n})\rightarrow 0,\ \ f(p_{n})\rightarrow c\}
\end{equation}
\noindent With the notations of section \ref{sec:CyclesEvInfini}, one may expect that $B_{\infty}(f)=T_{\infty}(f)$, see \cite[1.3]{T1} . 
\end{remark}

\section{Applications to Gauss-Manin systems and their Fourier transform}

\label{sec:F(GM)}

We study here the Gauss-Manin systems of regular functions.
As before,
let $U$ be an affine manifold of dimension $n\geq 2$, $S=\cit$ and $f:U\rightarrow S$ be a regular function. 

\subsection{Gauss-Manin systems of regular functions}
\label{sec:GaussManin}
Let $\Omega^{p}(U)$ be the space of regular $p$-forms on $U$. 
The Gauss-Manin complex of $f$ is
$(\Omega^{\bullet +n}(U)[\partial_{t}],d_{f})$
where $d_{f}$ is defined by
$$d_{f}(\sum_{i}\omega_{i}\partial_{t}^{i})=\sum_{i}d\omega_{i}\partial_{t}^{i}-\sum_{i}df\wedge \omega_{i}\partial_{t}^{i+1}$$
The Gauss-Manin systems of $f$ are the cohomology groups $M^{(i)}$ of this  complex. These are holonomic regular $\cit [t]<\partial_{t}>$-modules, see [Bo, p. 308], the action of $t$ and $\partial_{t}$ coming from the one on $\Omega^{\bullet +n}(U)[\partial_{t}]$ defined by
$$t(\sum_{i}\omega_{i}\partial_{t}^{i})=\sum_{i}f\omega_{i}\partial_{t}^{i}-\sum_{i}i\omega_{i}\partial_{t}^{i-1}$$
and
$$\partial_{t}(\sum_{i}\omega_{i}\partial_{t}^{i})=\sum_{i}\omega_{i}\partial_{t}^{i+1}$$
The following lemma is well-known \cite{DS}, \cite{DoSa1}, \cite{S}:
\begin{lemma}\label{lemma:Milibres}
Assume that $f$ has isolated singularities including at infinity.
Then the modules $M^{(i)}$ are $\cit [t]$ free of rank $h^{n-1+i}(U)$ for $i<0$.
\end{lemma}
\begin{proof} Follows from 
equation (\ref{eqH-iP}) and the fact that ${\cal H}^{-1}(^p {\cal H}^{i}(Rf_{*}\cit_{U}[n]))$
is a constant sheaf on $S$  for $i\neq n$ if $f$ has isolated singularities including at infinity, see section \ref{sec:SingIsolInfty}.
\end{proof} 

We will put $M:=M^{(0)}$ and we will call it  {\em the Gauss-Manin system} of $f$. Let $\widehat{M}$ be its Fourier transform: this is $M$ seen as a  $\cit[\tau] <\partial_{\tau}>$-module where $\tau$ acts as
$\partial_{t}$ and $\partial_{\tau}$ acts as $-t$:
$$\widehat{M}=\frac{\Omega^{n}(U)[\tau ]}{d_{f}(\Omega^{n-1}(U)[\tau])}$$
where $d_{f}(\sum_{i}\omega_{i}\tau^{i})=\sum_{i}d\omega_{i}\tau^{i}-\sum_{i}df\wedge \omega_{i}\tau^{i+1}$. 
Let
$$G:= \widehat{M}[\tau^{-1}]=\frac{\Omega^{n}(U)[\tau ,\tau^{-1} ]}{d_{f}(\Omega^{n-1}(U)[\tau ,\tau^{-1}])}$$
be the localized module.
Since $M$ is a regular holonomic $\cit [t]<\partial_{t}>$-module, $G$ is a free $\cit [\tau ,\tau^{-1}]$-module of finite rank equipped with a connection 
whose singularities are $0$ and $\infty$ only, the former being regular and the latter of Poincar\'e rank less or equal to $1$, see \cite[V, prop. 2.2]{S1}.

\begin{remark} \label{rem:FormelleSemiSimple}The Fourier-Laplace transform
$G$ has no ramification because $M$ is regular, see for instance \cite[V. 3. b.]{S1}. In particular $G$ has only integral slopes, the slopes $0$ and $1$. If $H$ is a lattice in $G$, i.e a free $\cit [\theta ]$-module of maximal rank, stable under $\theta^2 \partial_{\theta}$, 
the eigenvalues of the constant matrix in the  
expression of $\theta^2 \partial_{\theta}$ in a basis of $H$ are precisely the singular points of the Gauss-Manin system $M$, see \cite[V. 3]{S1}.
\end{remark}

The  {\em Brieskorn module} $G_0$ of $f$ 
is by definition the image in $G$ of the sections that do not depend on $\tau$. 
Putting $\theta :=\tau^{-1}$, we have
$$G_{0}:=\frac{\Omega^{n}(U)[\theta]}{d_{f}(\Omega^{n-1}(U)[\theta ,\theta^{-1}])\cap \Omega^{n}(U)[\theta ]}$$
where
$$d_{f}(\sum_{i}\omega_{i}\theta^{i})= \sum_{i}[d\omega_{i}\theta^{i+1}-df\wedge \omega_{i}\theta^{i}].$$ 
Recall the global Milnor number $\mu= \dim_{\cit}\frac{\Omega^{n}(U)}{df\wedge \Omega^{n-1}(U)}$.

\begin{proposition}\label{prop:G0SingIsolees}
Assume that $f$ has only isolated critical points on $U$. Then
\begin{enumerate}
\item $G_{0}=\frac{\Omega^{n}(U)[\theta ]}{(\theta d-df\wedge )\Omega^{n-1}(U)[\theta ]}$,
\item the $\cit [\theta]$-module $G_{0}$ has no torsion.
\end{enumerate}
\end{proposition}
\begin{proof} Both points follow from the generalized de Rham lemma: if $f$ has only isolated critical points on $U$, the cohomology groups of the complex $(\Omega^{\bullet}(U),\ df\wedge )$ 
all vanish, except perhaps the one in top degree which is equal to
$\frac{\Omega^{n}(U)}{df\wedge\Omega^{n-1}(U)}$.
\end{proof}

\begin{corollary}
Assume that $f$ has only isolated critical points $U$. If $G_{0}$ is free of finite type over $\cit [\theta]$ then it is free of rank $\mu$.
\end{corollary}
\begin{proof} Follows from proposition \ref{prop:G0SingIsolees}: 
$G_0$ is free because it has no torsion and the assertion about the rank is given by the formula
$\frac{G_{0}}{\theta G_{0}}=\frac{\Omega^{n}(U)}{df\wedge \Omega^{n-1}(U)}$.
\end{proof}

\begin{corollary}\label{coro:RangGG0}
Assume that $f$ has only isolated critical points $U$. Then 
\begin{enumerate}
\item the rank of $G$ over $\cit [\theta ,\theta^{-1}]$ is greater or equal than $\mu$,
\item  if $G_{0}$ is of finite type then $\Rang G =\mu $.
\end{enumerate}
\end{corollary}
\begin{proof}
1. Let $\omega_1 ,\cdots ,\omega_{\mu}$ be sections of $G_0$ whose classes are independent in $G_{0}/ \theta G_{0}$: by proposition
\ref{prop:G0SingIsolees} these sections are linearly independent over $\cit [\theta ]$ and generate a free 
submodule of rank $\mu$ in $G$. Because $G$ is free of finite type the assertion follows.   2. Under the assumption,  $G_0$ is free of 
rank $\mu$ and provides a lattice in $G$: the rank of $G$ is $\mu$.
\end{proof}

\noindent As a consequence, $G_0$ will not be of finite type if $\Rang G>\mu$. If $f$ has only isolated critical points we have $\Rang G\geq\mu$ and we get the following precisions
if $f$ has isolated singularities including at infinity:

\begin{theorem}\label{theo:rangMG}
If $f$ has at most isolated singularities including at infinity one has
\begin{equation}\label{eq:rangM}
\Rang M =\mu +\nu +h^{n-1}(U)-h^{n}(U)
\end{equation}
and
\begin{equation}\label{eq:rangG}
\Rang G =\mu +\nu 
\end{equation}
where $\mu$ is the global Milnor number of $f$.
\end{theorem}
\begin{proof} 
The rank of $M$ is $\dim_{\cit (t)}\cit (t)\otimes_{\cit [t]}M$, 
and this is also equal to the rank of
$\cit [t, p^{-1}(t)]\otimes_{\cit [t]}M$ where $p^{-1}(0)$ is the set of the singular points of  $M$.
Let ${\cal M}={\cal O}\otimes M$.
By formula (\ref{eq:H-1P}) one has 
$DR({\cal M})=(R^{n-1} f_* \cit_U )[1]$
and it follows that ${\cal M}_a$ is a free
${\cal O}_a$-module of rank 
$\dim H^{n-1}( f^{-1}(a), \cit )$
for $a\notin p^{-1}(0)$. Because ${\cal O}_a \otimes_{\cit [t]}M$ is isomorphic to 
$({\cal O}_a)^{\Rang M}$, the rank of $M$ is equal to $\dim H^{n-1}( f^{-1}(a), \cit )$ and
formula (\ref{eq:rangM}) follows from proposition \ref{prop:RangM}.
For formula (\ref{eq:rangG}), we use the exact sequence
$$...\rightarrow M^{(j)}\stackrel{\partial_{t}}{\rightarrow} M^{(j)}\rightarrow H^{n+j}(U,\cit )
\rightarrow\cdots\rightarrow H^{n-1}(U,\cit )\rightarrow M\stackrel{\partial_{t}}{\rightarrow}M\rightarrow H^{n}(U,\cit )\rightarrow 0$$
for $j\leq 0$. Because $f$ has isolated singularities including at infinity, it follows from lemma \ref{lemma:Milibres} that $\partial_{t}$ is surjective on $M^{(-1)}$
and this gives the exact sequence 
$$0\rightarrow H^{n-1}(U,\cit )\rightarrow M\stackrel{\partial_{t}}{\rightarrow}M\rightarrow H^{n}(U,\cit )\rightarrow 0$$
We also have
$$ \Rang G =\Rang M +\dim (\coker \partial_t )-\dim (\ker \partial_t)$$
see for instance \cite[Proposition V.2.2]{S1}, and the second formula follows from the first one.
\end{proof}

The converse of the second assertion of corollary \ref{coro:RangGG0}  is true, at least if $f$ 
 has at most isolated singularities including at infinity:

\begin{corollary}
Assume that $f$ has at most isolated singularities including at infinity and that $\Rang G =\mu $. Then $G_0$ is of finite type.
\end{corollary}
\begin{proof}
By theorem \ref{theo:rangMG} we have $\nu =0$ and thus $f$ is cohomologically tame: the result follows then from \cite{S}.
\end{proof}

\subsection{Basic example}
\label{sec:ExempleClassique}

We test the previous results on a classical wild example, see \cite{Br} for instance. Let
$f$ be defined on $\cit^{2}$ by  
$f(x,y)=y(xy-1)$.
It has no critical points at finite distance.

\begin{proposition}
\begin{enumerate}
\item $f$ has one singular point at infinity. The number $\nu$ of vanishing cycles at infinity is equal to
$1$ and $B(f)=\{0\}$.
\item  The $\cit [\tau ,\tau^{-1}]$-module $G$ is free of rank $1$ and the class $[dx\wedge dy]$ of $dx\wedge dy$ is a basis of it.
\end{enumerate}
\end{proposition}
\begin{proof} 
1. This result is well-known but we give the proof in order to test the notations of section \ref{sec:CompStandard}.
Homogeneization of the fibers of $f$ gives   
$$F(X_0 , X_1 , X_2 ,t)=X_1 X_2^{2}-X_2 X_0^{2}-tX_0^{3}=0$$
where the equation $X_0 =0$ defines the hyperplane at infinity. Notice that $\yit_{sing}=\{p\}\times \cit$ where
$p=(0:1:0)$ and in order to compute the number of vanishing cycles at infinity we can use corollary \ref{coro:nu}.
The Milnor number of the singularity $u^{2}-uv^{2}-tv^{3}=0$ at $(0,0)$
is equal to $2$ for all $t\neq 0$ and is equal to $3$ for $t=0$. 
The point $p$ is thus an isolated singular point of $f$ 
 at infinity for which $\nu_{p,0}=1$.\\
2. By 1. and theorem \ref{theo:rangMG},
we know that $G$ is free of rank $1$ over $\cit [\tau ,\tau^{-1}]$.
The differential form
$$\omega =rx^{r-1}y^{p}dx +px^{r}y^{p-1}dy,$$
 $r,p\geq 1$, is exact. We thus have $[df\wedge \omega] =0$ and 
\begin{equation}\label{eq:lemme1}
[(2r-p)x^{r}y^{p+1}dx\wedge dy]=[rx^{r-1}y^{p}dx\wedge dy]
\end{equation}
in $G$ for $r,p\geq 1$. 
An analogous computation shows that $[y^{p+1}dx\wedge dy]=0$ if $p\geq 1$ and that
$[2rx^{r}ydx\wedge dy]=[rx^{r-1}dx\wedge dy]$
if $r\geq 1$. 
If $2r\neq p$, one can express in particular $[x^r y^{p+1}dx\wedge dy]$ in terms of $[x^{r-1} y^{p}dx\wedge dy]$. If $2r=p$, notice that 
\begin{equation}\label{eq:lemme2}
\tau [x^r y^{2r+1}dx\wedge dy]=[x^r y^{2r}dx\wedge dy]
\end{equation}
Indeed, $df\wedge x^r y^{2r+1}dx =(-2x^{r+1}y^{2r+2}+x^r y^{2r+1})dx\wedge dy$ hence
$$(2r+1)[x^r y^{2r}dx\wedge dy]=2\tau [x^{r+1} y^{2r+2}dx\wedge dy]-\tau [x^r y^{2r+1}dx\wedge dy]$$
and we get formula (\ref{eq:lemme2}) using formula (\ref{eq:lemme1}).
This comptutation holds also for $r=0$, in particular $\tau [ydx\wedge dy]=[dx\wedge dy]$.
Last,
$$\tau^{-1}[x^{q}dx\wedge dy]=[df\wedge\frac{x^{q+1}}{q+1}dy] =[y^{2}\frac{x^{q+1}}{q+1}dx\wedge dy]=[a_{q}x^{q}ydx\wedge dy]=[b_{q}x^{q-1}dx\wedge dy]$$
for $q\geq 1$, where $a_{q}$ and $b_{q}$ are non zero constant, as shown by formula (\ref{eq:lemme1}). 
This shows that one can express the class of any form in terms of $[dx\wedge dy]$, which is thus a generator of $G$. 
\end{proof}


\section{Hypersurfaces in weighted projective spaces}

In this section we recall basic results about hypersurfaces in weighted projective spaces.
We will consider only {\em smooth} hypersurfaces and the goal of this section is to give a characterization of such objects, see theorem \ref{theo:hyperlisse}.
 Our references are \cite{Dimca}, \cite{Dolgachev} and \cite{Fletcher}.

\subsection{Smooth hypersurfaces in weighted projective spaces}
\label{sec:SingEPP}
Let $w_0 ,\cdots ,w_n$ and $d$ be integers greater than zero. In what follows, except otherwise stated, we will assume that $n\geq 3$ and 
that the weights $w_i$ 
are normalized, that is              
\begin{equation}\nonumber
\gcd (w_0 ,\cdots ,\hat{w}_i ,\cdots ,w_n)=1\ \mbox{for all}\ i=0,\cdots ,n\ \mbox{and}\ w_0 \leq w_{1}\leq\cdots \leq w_n
\end{equation}

\noindent Recall that a polynomial $W$ is quasi-homogeneous of weight $(w_0 ,\cdots ,w_n)$ and of degree $d$ if 
$$W(\lambda^{w_0}u_0,\cdots ,\lambda^{w_n}u_n)=\lambda^d W(u_0,\cdots ,u_n)$$
for any non zero $\lambda$. Equation $W(u_0 ,\cdots , u_n )=0$ defines a hypersurface $H$ (resp. $CH$) of degree $d$ in the weighted projective space 
$\ppit (w):=\ppit (w_0,\cdots ,w_n)$ (resp. $\cit^{n+1}$). 
The hypersurface  $H$ is {\em quasi-smooth} if $CH - \{0\}$ is smooth. 

\begin{remark}
If $d=w_i$ for some index $i$ then  $W=a_i u_i +g(u_0 ,\cdots , \widehat{u_i } ,\cdots , u_n)$, where $a_i \in\cit^*$ and $g$ is quasi-homogeneous. The hypersurface
$H$ is then isomorphic to the weighted projective space $\ppit (w_0 ,\cdots , \widehat{w_{i}}, \cdots ,w_n )$ via 
$$(u_0 ,\cdots ,\widehat{u_i},\cdots ,u_{n})\mapsto (u_0 ,\cdots , -a_{i}^{-1}g(u_{0},\cdots ,\widehat{u_{i}},\cdots ,u_n ),\cdots , u_n)$$
In this situation, we will say that $H$ is {\em a linear cone}.
\end{remark}

For $x\in\ppit (w_0 ,\cdots ,w_{n})$, let
$I(x)=\{j, \ x_j \neq 0\}$
and, for prime $p$,
$$Sing_p (\ppit (w))=\{x\in\ppit (w),\ p \ \mbox{divides}\  w_i \ \mbox{for all}\ i\in I(x)\}$$
The singular locus $\ppit_{sing}(w_{0},\cdots ,w_{n})$ (or $\ppit_{sing}(w)$) of $\ppit (w_0 ,\cdots ,w_{n})$ is 
\begin{equation}\label{eq:SingPw}
\ppit_{sing}(w)=\cup_{p\ prime}Sing_p (\ppit (w))
\end{equation}
The hypersurface $H$ is in {\em general position} with respect to $\ppit_{sing} (w)$ (for short: in general position) if
\begin{equation}\nonumber
\codim_{H}(H\cap \ppit_{sing}(w))\geq 2
\end{equation}
A hypersurface in general position inherits the singularities of the ambient space:

\begin{proposition}\cite[Proposition 8]{Dimca}
\label{prop:adjonction}
The singular locus of a quasi-smooth hypersurface $H$ in general position is $H_{sing}=H\cap \ppit_{sing}(w)$.
\qed
\end{proposition}

\noindent Assume that the degree $d$ hypersurface $H$ is in general position and quasi-smooth. Then
$\omega_{H}\simeq {\cal O}_{\ppit (w)}(d-\sum_{i=0}^{n}w_{i})_{|H}$
where $\omega_{H}$ denotes the canonical bundle, see \cite[Theorem 3.3.4 and Theorem 3.2.4]{Dolgachev} . Put
$w :=\sum_{i=0}^{n}w_i$.
Under the assumptions of proposition \ref{prop:adjonction}, we will say that $H$ 
is {\em Fano} if $d<w$ and {\em Calabi-Yau} if $d=w$.
We will mainly consider the Fano case.

\begin{theorem}\label{theo:hyperlisse}
 Let $H$ be a degree $d$ hypersurface in $\ppit (w_{0},\cdots ,w_{n})$. Assume that
 \footnote{The first and the second conditions imply the third except when $H$ is a degree $d$
hypersurface in $\ppit (1,\cdots , 1, d)$: 
 the purpose of the third condition is to remove the linear cones. This will simplify the statements.}
\begin{enumerate}
\item $\gcd (w_i ,w_j )=1$ for all $i,j$,
\item $w_i$ divides $d$ for all $i$,
\item $w_{i}<d$ for all $i$.
\end{enumerate}
\noindent Then $H$ is not a linear cone, is in general position, quasi-smooth and smooth. 
\end{theorem}
\begin{proof} By \cite[I.3.10]{Fletcher},
a degree $d$ hypersurface is in general position if and only if
$$\gcd (w_{0},\cdots ,\hat{w}_{i},\cdots , \hat{w}_{j},\cdots ,w_{n})|d$$
for all $i,j$, $i\neq j$, and
 $$\gcd (w_{0},\cdots ,\hat{w}_{i},\cdots ,\cdots ,w_{n})=1$$
for all $i$. Therefore the first condition
shows that $H$ is in general position. 
The second condition shows that $H$ is quasi-smooth, see \cite[Theorem I.5.1]{Fletcher}. 
Last, and in order to show that $H$ is smooth we use the following numerical criterion \cite{Dimca}: 
for any prime $p$, let us define                 
\begin{equation}\nonumber
m(p)=card \{i; p\ \mbox{divides}\ w_i\},\ k(p)=1\ \mbox{if}\ p\ \mbox{divides}\ d,\ 0\ \mbox{otherwise},\ q(p)=n-m(p)+k(p)
\end{equation}
Then the quasi-smooth and in general position degree $d$ hypersurface $H$ is smooth if and only if 
$q(p)\geq n $
for any prime $p$.
The first condition shows that $m(p)\leq 1$: if $m(p)=0$ we get, by the very definition, $q(p)\geq n$; 
if $m(p)=1$ the second condition shows that $k(p)=1$ and thus $q(p)= n$.
\end{proof}

\begin{example}(Surfaces)\label{ex:Surfaces} 
 The degree $6$ hypersurface in $\ppit (1,1,2,3)$ is in general position and smooth. It is a Fano surface. 
The other smooth Fano surfaces are the surfaces of degree $2$
or $3$ in $\ppit (1,1,1,1)$ and surfaces of degree $4$ in $\ppit (1,1,1,2)$.
\end{example}

\begin{remark}(Curves)
The previous results have been established for $n\geq 3$. 
If $H$ is a curve of degree $d$ in $\ppit (w_{0},w_{1},w_{2})$ then $H$ is in general position, is smooth and is not a linear cone if and only if
the conditions of theorem \ref{theo:hyperlisse}
are satisfied \cite[Theorem II.2.3]{Fletcher}.
\end{remark}

\subsection{The quantum differential equation of a smooth hypersurface in a weighted projective space}
\label{sec:ODQ}
 Let $H$ be a degree $d$ smooth hypersurface in the weighted projective
space $\ppit (w_{0},\cdots ,w_{n})$. The differential operator  
\begin{equation}\label{eq:OpDiffQuant}
P_{H}=
\prod_{i=0}^{n}[(w_i \theta q\partial_q )(w_i \theta q\partial_q -\theta)\cdots (w_i \theta q\partial_q -(w_i -1)\theta )]
-q(d \theta q\partial_q +\theta )\cdots (d \theta q\partial_q +d\theta )
\end{equation}
 is called the {\em quantum differential operator} of $H$
($q$ is the quantization variable). We will often write $P$ instead of $P_{H}$.
The key point is that the quantum differential equation $P_H =0$, which depends only on combinatorial data, can be used in order to describe the small quantum cohomology of
the $H$, 
see for instance \cite{CG} and section \ref{sec:PCQ}.

\begin{proposition}\label{prop:rangGKZ}
Under the assumptions of theorem \ref{theo:hyperlisse}, the module
\begin{equation}\label{eq:GKZ}
M_{A}:=\cit [\theta , q, q^{-1}]<\theta q\partial_{q}>/ \cit [\theta , q, q^{-1}]<\theta q\partial_{q}> P_H
\end{equation}
is a free $\cit [\theta , q, q^{-1}]$-module of rank $n$. 
\end{proposition}
\begin{proof}
Notice first that, using the relation $\partial_q q = q\partial_q +1$,
the equation $P_H =0$ takes the form        
\begin{equation}\nonumber
\theta^{\mu}\prod_{i=0}^{n}w_{i}^{w_{i}}\prod_{i=0}^{n}(q\partial_q )(q\partial_q -\frac{1}{w_i})\cdots (q\partial_q -\frac{w_i -1}{w_i})
=\theta^{d}d^d (q\partial_q )(q\partial_q -\frac{1}{d} )\cdots (q\partial_q -\frac{d-1}{d})q
\end{equation}
By assumption, $w_i$ divides $d$: we write $d=m_i w_i$ and we define     
$$v_i :=card \{k\in \{1,\cdots , d-1\}; \ m_i \ \mbox{divides}\ k \}$$
for $i=0,\cdots ,n$. 
Let $k\in \{1,\cdots ,d-1\}$. If $m_i$ divides $k$, write $k=m_i\ell_i$ : we have $d\ell_i = kw_i$ and thus
$\frac{\ell_i}{w_i}=\frac{k}{d}$. Conversely, if there exists $k\in \{1,\cdots ,d-1\}$ such that $\frac{\ell_i}{w_i}=\frac{k}{d}$ then $k=m_{i}\ell_{i}$.
After cancellation of the common factors on the left and on the right of the equation, the quantum differential 
operator $P_H$ is of degree $w_0 +\cdots + w_n -1-\sum_{i=0}^{n}v_i$ in $q\partial_q$. 
If $d=w_{1}\cdots w_{n}$ we have
$v_{i}=w_{i}-1$ for $i=1,\cdots ,n$ and the proposition follows because the rank of $M_A$ is the degree of the irreducible polynomial $P$ in 
$\theta q\partial_q$.
\end{proof}

\section{Givental-Hori-Vafa models of smooth hypersurfaces in weighted projective spaces}

\label{sec:ModelesHV}
We define here, following \cite{Giv0}, \cite{Giv} and \cite{HV}, mirror partners
for the small quantum cohomology of smooth hypersufaces in weighted projective spaces.
Let $H$ be a degree $d$ hypersurface in $\ppit (w_0,\cdots ,w_n)$. 
Otherwise stated, we assume that $H$ is Fano.

\subsection{Givental-Hori-Vafa models as Laurent polynomials}
\label{sec:var}
The Givental-Hori-Vafa model of $H$ (for short: GHV model)
is the function $f(v_0,\cdots ,v_n )=v_0 +v_1 +\cdots +v_n$ on the variety $U$ defined by the equations
\begin{align}\label{eq:ModHV}
 \left\{ \begin{array}{l}
v_0^{w_0}\cdots v_n^{w_n}=q\\
\sum_{j\in J}v_j =1
\end{array}
\right .
\end{align}
\noindent where $(v_0, \cdots ,v_n, q)\in (\cit^*)^{n+2}$ and $J$ is a set of indices such that $\sum_{j\in J}w_j =d$.
The variable $q$ is the quantization variable.

\begin{proposition}
\label{prop:HVLaurent}
Under the assumptions of theorem \ref{theo:hyperlisse}, 
one may assume that       
\begin{equation}\nonumber
w_0 =1\ \mbox{and}\  d=w_{r+1}+\cdots +w_n 
\end{equation}
for some $r\in\{0,\cdots ,n-2\}$. In these conditions, and up to the ramification $q=Q^{w_n}$,
the GHV model of $H$ takes the form
\begin{equation}\label{eq:PolLaurBis}
f(u_1 ,\cdots ,u_{n-1},Q)=u_1 +\cdots + u_r +\frac{(u_{r+1}+\cdots +u_{n-1}+Q)^d}{u_1^{w_1}\cdots u_{n-1}^{w_{n-1}}}
\end{equation}
for $(u_1, \cdots , u_{n-1}, Q)\in (\cit^*)^{n}$.
\end{proposition}

\begin{proof} By \cite[Theorem 9]{P}, and under the assumptions of theorem \ref{theo:hyperlisse},
one may assume that       
$w_0 =1$ and  $d=w_{r+1}+\cdots +w_n $ and that the GHV model of $H$ takes the form, if $r\geq 1$,
\begin{equation}\label{eq:PolLaur}
f(x_1 ,\cdots ,x_{n-1})=x_1 +\cdots + x_r +1+q\frac{(x_{r+1}+\cdots +x_{n-1}+1)^d}{x_1^{w_1}\cdots x_{n-1}^{w_{n-1}}}
\end{equation}
where $(x_1 ,\cdots ,x_{n-1}, q)\in (\cit^*)^{n}$.
Removing the harmless additive constant $1$, we get (\ref{eq:PolLaurBis}) if we put $u_{i}=q^{1/w_{n}}x_{i}$
for $i=r+1,\cdots , n-1$ and $u_{i}=x_{i}$ for $i=1,\cdots, r$. 
 The formula is obviously adapted for $r=0$. 
\end{proof}

\noindent Notice
the following homogeneity relation
\begin{equation}\label{eq:fHom}
f=\frac{w -d}{w_n }\ Q\frac{\partial f}{\partial Q}+ \sum_{i=1}^{r}u_{i}\frac{\partial f}{\partial u_{i}}
+\frac{w -d}{w_{n}}\sum_{i=r+1}^{n-1}u_{i}\frac{\partial f}{\partial u_{i}}
\end{equation} 
The function $f$ is homogeneous of degree $1$ if     
\begin{equation}\label{eq:degui}
u_1 ,\cdots ,u_r\ \mbox{have degree}\ 1
\end{equation}
and 
\begin{equation}\label{eq:degq}
u_{r+1} ,\cdots ,u_{n-1}\ \mbox{et}\ Q\ \mbox{have degree}\ \frac{w -d}{w_n }.
\end{equation}

\begin{lemma}\label{lemma:wiPR} 
Under the assumptions of theorem \ref{theo:hyperlisse}, we have
$n+d>w$.
\end{lemma}
\begin{proof} 
By \cite[Proposition 7]{P} and under the assumptions of theorem \ref{theo:hyperlisse}, there are at least
$w -d +1$ weights $w_{i}$ equal to $1$. In particular $w-d+1\leq n+1$ and we get $n+d\geq w$. Because $w-d+1=1+w_0 +\cdots +w_r$ we also deduce that there are at least $r+2$ weights equal to $1$: therefore, if $n+d=w$ we have $n=w_0 +\cdots +w_r=r+1$ and this is not possible because $r\leq n-2$. This gives the assertion.
\end{proof}

\noindent It follows that $n+d-w>0$ and we will see in sections \ref{sec:Conjectures} and \ref{sec:BirHV} that this number is a potential number of vanishing cycles at infinity.\\

Last, the two next results follow from straightforward computations: we fix $Q\in\cit^*$ and we  denote by $f^{o}$ the Laurent polynomial (\ref{eq:PolLaurBis}) and by $Q_{f}^{o}$ its Jacobian ring.

\begin{lemma}\label{lemma:Critf} 
\begin{enumerate}
\item The Laurent polynomial $f^{o}$ has $w -d$ isolated
critical points on $(\cit^* )^{n-1}$,
\begin{equation}\nonumber
 c_{k}=(b_1 \varepsilon^{k},\cdots ,b_r \varepsilon^{k}, \frac{w_{r+1}}{w_{n}}Q ,\cdots ,\frac{w_{n-1}}{w_{n}}Q )
\end{equation}
for $k=0,\cdots , w -d-1$, where $\varepsilon$ denotes a $w -d$-th primitive root of the unity and 
$b_{i}=w_{i}(\frac{d^d}{w_{1}^{w_{1}}\cdots w_{n}^{w_{n}}}Q^{w_n})^{1/(w -d)}$ for $i=1,\cdots ,r$. 
\item These critical points are nondegenerate and 
the corresponding critical values are
\begin{equation}\nonumber
f^{o}(c_{k})=(w-d)(Q^{w_n} \frac{d^d}{w_{1}^{w_{1}}\cdots w_{n}^{w_{n}}})^{1/(w -d)}\varepsilon^{k}
\end{equation}
for $k=0,\cdots , w -d-1$.
\end{enumerate}\qed
\end{lemma}

\begin{corollary}\label{coro:VPf}
\begin{enumerate}
\item The eigenvalues of the multiplication by $f^{o}$ on $Q_{f}^{o}$ are pairwise distinct.
\item The classes of $1,f^{o},\cdots , (f^{o})^{w-d-1}$ provide a basis of $Q_{f}^{o}$.
\item $(f^{o})^{w-d}=(w-d)^{w-d}Q^{w_n} \frac{d^{d}}{w_{1}^{w_{1}}\cdots w_{n}^{w_{n}}}$.
\end{enumerate}\qed
\end{corollary}

\subsection{Singular points at infinity of GHV models}
\label{sec:Conjectures}

 How to compute the rank of the connection $G$ associated with a GHV model $f$? The hope is to use theorem
\ref{theo:rangMG}: the main question is to decide whether $f$ has
at most isolated singularities including at infinity for a suitable compactification or not. This problem is not so easy in general\footnote{There exist different theoretic classes of functions having isolated singularities including at infinity in some sense, see for instance 
the book \cite{T1} and the references therein. But in general it is not clear how to decide if a given function belongs to a class or to another.}.
Let us begin with the following examples:

\begin{example}
\label{sec:Q2P3}
We keep the notations of remark \ref{rem:CompStandardLaurent}.
\begin{enumerate}
\item Let us consider the GHV model\footnote{We fix here $q=1$.} of a smooth hypersuface of degree $2$ in $\ppit^3$:
\begin{equation}\nonumber
f(x,y)=x+\frac{(y+1)^2}{xy}
\end{equation}
The equation $F(X_0 , X_1 , X_2 ,t)=0$ is
$X_1^{2}X_2 -tX_0 X_1 X_2 + X_0 (X_2 +X_0 )^{2}=0$
and $\yit_{0}$ has an isolated singular point at $P\times \{0\}$ where $P=(1:0:-1)$ is on the polar locus at finite distance. 
In order to compute the number of vanishing cycles $\nu_{P, 0}$, we use proposition \ref{prop:nu}: the hypersurface 
$$u^2 v-tuv+(1+v)^{2}=0$$
is smooth for small $t\neq 0$ but the Milnor number at $P$ for $t=0$ is $\mu_{P,0}=1$. Thus $\nu_{P,0}=\mu_{P,0}=1$. 
The value $t=0$ is atypical. By theorem \ref{theo:rangMG}, the rank of $G$ is $w-d+\nu =2+1=3$. 
\item Let us now consider the GHV model of a smooth hypersurface of degree $2$ in $\ppit^4$:
\begin{equation}\nonumber
f(x,y,z)=x+y+\frac{(z+1)^2}{xyz}
\end{equation}
Equation $F(X_0 ,X_1 ,X_{2}, X_{3}, t)=0$ takes the form
$$X_1^{2}X_2 X_3+X_1 X_2^2 X_3 +X_{0}^{2}(X_{3}+X_{0})^2-tX_0 X_1 X_2 X_3 =0$$
and $\yit_{0}$ has an isolated singular point at $P\times \{0\}$ where $P=(1:0:0:-1)$ is on the polar locus at finite distance. 
It follows from proposition \ref{prop:nu} that $\nu_{P,0}=4-3=1$:
$P$ is thus a singular point at infinity,
the value $t=0$ 
is atypical and the number of vanishing cycles at infinity is $1$ (the points on the hyperplane at infinity do not produce vanishing cycles). 
By theorem \ref{theo:rangMG},
the rank of $G$ is therefore $4$.          
\end{enumerate}
\end{example}

\begin{example}
\label{ex:MalBr}
The GHV model of a degree $d\geq 2$ Fano hypersurface in $\ppit^n$ is
\begin{equation}\nonumber
f(u_1 ,\cdots ,u_{n-1})=u_1 +\cdots + u_{n-d} +\frac{(u_{n-d+1}+\cdots +u_{n-1}+q)^d}{u_1 \cdots u_{n-1}}
\end{equation}
Recall the set $T_{\infty}(f)$ defined in remark \ref{rem:NonCrit}.
We have $T_{\infty}(f)=\{0\}$:
indeed, let us define the sequence $(u_{p})=((u_1^p,\cdots ,u_{n-1}^p))$ by
$$u_{1}^{p}=\cdots =u_{n-d}^{p}=\frac{1}{p}\  \mbox{and}\  u_{n-d+1}^p =\cdots =u_{n-1}^p=\frac{1}{p^{n-d+1}}-\frac{q}{d-1}$$
Then  
$$u_{p}\rightarrow (0,\cdots ,0,-\frac{q}{d-1},\cdots ,-\frac{q}{d-1}),\ u_{p}\grad f(u_{p})\rightarrow 0\ \mbox{et}\ f(u_{p})\rightarrow 0$$
thus $\{0\}\subset T_{\infty}(f)$ and there are no other candidates in $T_{\infty}(f)$. This suggests that $0$ is an atypical value and that
$P=(0,\cdots ,0,-\frac{q}{d-1},\cdots ,-\frac{q}{d-1})$ is a singularity at infinity.
\end{example}

We thus propose the following conjecture (recall that $n+d-w >0$, see lemma \ref{lemma:wiPR}):           

\begin{conjecture}
\label{conj:SingInftyHV}
Under the assumptions of theorem \ref{theo:hyperlisse}, there exists a compactification for which GHV models have only one 
singular point $P$ at infinity, located on the polar locus at finite distance, for which
$\nu =\nu_{P,0} =n+d-w$.
\end{conjecture}

\begin{corollary} 
\label{coro:RangG}
Under the assumptions of theorem \ref{theo:hyperlisse},
the rank of $G$ is equal to $n$.
\end{corollary}
\begin{proof} Follows from theorem \ref{theo:rangMG} and lemma \ref{lemma:Critf} 
and conjecture \ref{conj:SingInftyHV}.
\end{proof}

\section{Application to mirror symmetry for smooth hypersurfaces in projective spaces. The case of the quadrics}
\label{sec:BirHV}

We explain in this section why the GHV models 
should be mirror partners of smooth hypersurfaces in weighted projective spaces.
The general setting is described in section \ref{sec:ProbBir} and we apply it to quadrics in $\ppit^n$ in section \ref{sec:QuadPnB}.
We work under the assumptions of theorem \ref{theo:hyperlisse}.

\subsection{Mirror symmetry  and the Birkhoff problem}

\label{sec:ProbBir}

Let $H$ be a smooth Fano degree $d$ hypersurface in the weighted projective space $\ppit (1,w_{1},\cdots ,w_{n})$. 
Its GHV model is 
\begin{equation}\nonumber
f(u_1 ,\cdots ,u_{n-1}, Q)=u_1 +\cdots + u_r +\frac{(u_{r+1}+\cdots +u_{n-1}+Q)^d}{u_1^{w_1}\cdots u_{n-1}^{w_{n-1}}}
\end{equation}
see section \ref{sec:ModelesHV}. We will use freely the notation $q=Q^{w_n}$. 
The connection $G$ is defined as in section 
$\ref{sec:F(GM)}$: it is a free $\cit [\theta , \theta^{-1}, Q, Q^{-1}]$-module equipped with a connection $\nabla$ defined by 
\begin{equation}\nonumber
\theta^2 \nabla_{\partial_{\theta}}  [\sum_{i}\omega_{i}\theta^{i}]=[\sum_{i}f\omega_{i}\theta^{i}]-
[\sum_{i}i\omega_{i}\theta^{i+1}]
\end{equation}
and
\begin{equation}\nonumber
\theta  \nabla_{Q\partial_{Q}} [\sum_{i}\omega_{i}\theta^{i}]=[\sum_{i}Q\partial_Q (\omega_{i})\theta^{i+1}]
-[\sum_{i}Q\frac{\partial f}{\partial Q}\omega_{i}\theta^{i}]
\end{equation}

\noindent where the $\omega_i$'s are differential forms on $(\cit ^*)^{n-1}\times \cit^* $, equipped with coordinates $(u_1 ,\cdots ,u_{n-1}, Q)$,
$Q\partial_Q (\omega_{i})$ denotes the Lie derivative of the differential form $\omega_{i}$ in the direction of $Q\partial_Q$ and $[\ ]$ 
denotes the class in $G$.

The general principle of the mirror symmetry considered in this paper is to associate to the GHV model
$f$ a differential system isomorphic to the one of the small quantum cohomology of the hypersurface $H$, see sections 
\ref{sec:PCQPoids} and \ref{sec:PCQPoids}.
 This can be done solving the following Birkhoff problem:
find a free $\cit [Q,\theta ]$-module $H_0 ^{log}$ of rank $n$ in $G$ and a basis $(\omega_0 ,\cdots , \omega_{n -1})$ of it in which the matrix of 
the flat connection
$\nabla$ takes the form
\begin{equation}\label{eq:Bir}
(\frac{A_0 (Q)}{\theta}+A_1 (Q) )\frac{d\theta}{\theta}+(\frac{\Omega_{0}(Q)}{\theta}+\Omega_{1}(Q))\frac{dQ}{Q}
\end{equation} 
and such that:
\begin{itemize}
\item  $A_1 (Q)$ is semi-simple, 
with eigenvalues $0,1,\cdots ,n-1$ (and, up to a factor $2$, this corresponds to cohomology degrees),
\item $P_H (-\theta\nabla_{q\partial_q}, q, -\theta ) (\omega_0 )=0$
where $P_H$ is the quantum differential operator  defined in section \ref{sec:ODQ}.
\end{itemize}
The rank of $H_0^{log}$ is equal to $n$ because of proposition \ref{prop:rangGKZ}. 
Notice that the expected module $H_0^{log}$ is not the Brieskorn module $G_{0}$, because the latter is not of finite type.

In formula (\ref{eq:Bir}), there is a relation between the matrices $A_0 (Q)$ and $\Omega_{0}(Q)$:
let
$(a_1 ,\cdots , a_{n-1})\in \zit^{n-1}$, 
\begin{equation}\label{eq:defFormePrim}
\omega_0 =\frac{d u_{1}}{u_{1}}\wedge\cdots\wedge \frac{d u_{n-1}}{u_{n-1}}
\end{equation}
and $ [u_{1}^{a_{1}}\cdots u_{n-1}^{a_{n-1}}\omega_0 ]$ be the class of               
$u_{1}^{a_{1}}\cdots u_{n-1}^{a_{n-1}}\omega_0$
in $G$; then 

\begin{lemma}\label{lemma:var}
 We have
\begin{equation}\nonumber
\theta^{2}\nabla_{\partial_{\theta}}[u_{1}^{a_{1}}\cdots u_{n-1}^{a_{n-1}}\omega_{0} ]=
-(w -d)\theta\nabla_{\frac{1}{w_n}Q\partial_Q} [u_{1}^{a_{1}}\cdots u_{n-1}^{a_{n-1}}\omega_{0} ]
\end{equation}
\begin{equation}\nonumber 
+(\sum_{i=1}^{r}a_{i}+\frac{w -d}{w_{n}}\sum_{i=r+1}^{n-1}a_{i})\theta [u_{1}^{a_{1}}\cdots u_{n-1}^{a_{n-1}}\omega_{0}]
\end{equation}
in $G$. 
\end{lemma}
\begin{proof}
Follows from (\ref{eq:fHom}) and the definition of $\nabla$.
\end{proof}

 As explained in remark \ref{rem:relationPCQ},
the matrix $-\Omega_{0}(q)$ (or, up to a constant, the matrix $A_0 (q)$) should provide the characteristic relation
\begin{equation}\label{eq:RelCohQuantPoids}
b^{\circ n} =q\frac{d^d}{\prod_{i=0}^{n}w_{i}^{w_{i}}} b^{\circ d+n-w}
\end{equation}
 in small quantum cohomology, see equation (\ref{eq:RelCohQuantPoidsEx}), 
where $\circ$ denotes the  quantum product, $b$ the hyperplane class and $w=w_0 +\cdots + w_n$, \cite{CG}, \cite{GuestSakai}, 
see section \ref{sec:PCQPoids} for details. 
 Therefore it deserves a particular study. 
Recall that $n+d-w>0$, see lemma \ref{lemma:wiPR}.
Let $P_c (A_0 )$ be the characteristic polynomial of $A_0$. 

\begin{proposition} Assume that the rank of $G$ 
is equal to $n$. Then
\begin{equation}\label{eq:PolCarA0Gen}
 P_{c}(A_{0})(\zeta , Q)=P_{c}^{fin}(A_{0})(\zeta , Q)P_{c}^{\infty}(A_{0})(\zeta , Q)
\end{equation}
where
\begin{equation}\nonumber 
P_{c}^{fin}(A_{0})(\zeta , Q)=\zeta^{w-d}-(w-d)^{w-d}\frac{d^d}{\prod_{i=0}^{n}w_{i}^{w_{i}}}Q^{w_n}
\end{equation}
and
\begin{equation}\nonumber 
P_{c}^{\infty}(A_{0})(\zeta , Q)=\zeta^{n+d-w}+\sum_{i,j\geq 0}a_{i,j}\zeta^{i}Q^{w_n j}\ \mbox{with}\  i+(w-d)j=n+d-w
\end{equation}
\end{proposition}
\begin{proof}
The characteristic polynomial
$P_{c}(A_{0})(\zeta , Q)$ is homogeneous of degree $n$, where $Q$ is homogeneous of degree $(w-d)/w_n$
and $\zeta$ is of degree $1$, see section \ref{sec:var}.
Therefore, equation (\ref{eq:PolCarA0Gen}) follows from lemma \ref{lemma:Critf} and remark \ref{rem:FormelleSemiSimple}.  
\end{proof}

\begin{corollary} Assume that the rank of $G$ 
is equal to $n$. We have
\begin{equation}\label{eq:PolCarA0}
P_{c}(A_{0})(\zeta , Q)=\zeta^{n}-(w-d)^{w-d}\frac{d^d}{\prod_{i=0}^{n}w_{i}^{w_{i}}}Q^{w_n}\zeta^{n+d-w}
\end{equation}
if and only if $B_{\infty}(f)=\{0\}$, see section \ref{sec:CyclesEvInfini}.\qed
\end{corollary}

\subsection{Illustration: smooth quadrics in $\ppit^{n}$}

\label{sec:QuadPnB}

The aim of this section is to test the previous discussions for quadrics in $\ppit^{n}$. This paragraph has been inspired by \cite{GS}, 
which deals in a slightly different way with quadrics in $\ppit^{4}$. 
We prove in particular the theorem announced in the introduction, see section \ref{sec:ProofTheoPpl}.

\subsubsection{The GHV model of a quadric}

The GHV model of a quadric in $\ppit^n$ is 
\begin{equation}\nonumber
f(u_{1}, \cdots , u_{n-1})=u_1 +\cdots +u_{n-2} +\frac{(u_{n-1}+q)^{2}}{u_{1}\cdots u_{n-1}}
\end{equation}
see section \ref{sec:var} (we have now $Q=q$ with the previous notations). Recall the Brieskorn module $G_0$ defined as in section \ref{sec:GaussManin}.

\subsubsection{The Birkhoff problem}
\label{sec:BirkhoffQuadric}
Let $\omega_{0}=\frac{du_{1}}{u_{1}}\wedge \cdots\wedge \frac{du_{n-1}}{u_{n-1}}$ and
\begin{equation}\label{eq:baseBir}
\varepsilon :=([\omega_{0}], [u_1 \omega_{0}],\cdots , [u_{1}\cdots u_{n-2}\omega_0 ], 2[u_{n-1}\omega_{0}])
:=(\varepsilon_{0}, \cdots , \varepsilon_{n-1})
\end{equation}
where $[\alpha ]$ denotes the class of $\alpha$ in $G$. 
One has
$$\frac{G_{0}}{\theta G_{0}}=\frac{\Omega^{n}(V)[q, q^{-1}]}{df\wedge \Omega^{n-1}(V)[q,q^{-1}]}$$
where the differential $d$ is taken with respect to $u\in V:= (\cit^* )^{n-1}$, see proposition \ref{prop:G0SingIsolees}.

\begin{lemma}\label{lemma:AnneauJacLibre}
The $\cit [q,q^{-1}]$-module $G_{0}/\theta G_{0}$ is free of rank
 $n-1$ with basis
\begin{equation}\label{eq:BaseJac}
([\omega_{0}], [u_1 \omega_{0}],\cdots , [u_{1}\cdots u_{n-2}\omega_0 ])
\end {equation}
\end{lemma}
\begin{proof}
Let us show that (\ref{eq:BaseJac}) gives a system of generators. Notice first the relations   
\begin{equation}\label{eq1}
u_{i}\frac{\partial f}{\partial u_{i}}=u_1 -\frac{(u_{n-1} +q)^2}{u_1 \cdots u_{n-1}}
\end{equation}
for $i=1,\cdots ,n-2$ and
\begin{equation}\label{eq2}
u_{n-1}\frac{\partial f}{\partial u_{n-1}}=\frac{u_{n-1}^{2} -q^2}{u_1 \cdots u_{n-1}}
\end{equation}
We thus have
\begin{equation}\nonumber
u_{1}\frac{\partial f}{\partial u_{1}}-u_{n-1}\frac{\partial f}{\partial u_{n-1}}=u_{1}-2\frac{u_{n-1}+q}{u_1 \cdots u_{n-2}}
\end{equation}
from which we get 
\begin{equation}\label{eq3}
u_{n-1} +q =\frac{1}{2}u_{1}^{2}u_2 \cdots u_{n-2}\ mod\ (\partial_{u_1}f,\cdots , \partial_{u_n}f)
\end{equation}
Putting this in (\ref{eq1}), we get                    
\begin{equation}\nonumber
u_{n-1} =\frac{1}{4}u_{1}^{2}u_2 \cdots u_{n-2}\ mod\ (\partial_{u_1}f,\cdots , \partial_{u_n}f)
\end{equation}
and, using (\ref{eq3}),
\begin{equation}\nonumber
u_{n-1} =q \ mod\ (\partial_{u_1}f,\cdots , \partial_{u_n}f)\ \mbox{and}\ u_{1}^{2}u_2 \cdots u_{n-2}=4q\ mod\ (\partial_{u_1}f,\cdots , \partial_{u_n}f)
\end{equation}
We deduce from this that we have indeed a system of generators because 
\begin{equation}\nonumber
u_1 =u_{2}=\cdots =u_{n-2}
\end{equation}
(always modulo the Jacobian ideal).
This gives in particular the relations
\begin{equation}\label{eq:RelAlgJacobi}
 u_{n-1}\omega_0 =q\omega_0\ \mbox{et}\  u_{n-2}\omega_0 =\cdots =u_{1}\omega_0
\end{equation}
in $G_{0}/\theta G_{0}$. 
Last, corollary \ref{coro:VPf}
shows that there are no non trivial relations between the sections:
for $i= 1,\cdots , n-2$, 
the classes of $u_{1}\cdots u_{i}\omega_{0}$ and $f^{i}\omega_{0}$ are indeed proportional in $G_0 /\theta G_0$.
\end{proof}

Let us now define
\begin{itemize}
\item $H$ ({\em resp.} $H^{log}$) the sub-$\cit [\theta ,\theta^{-1}, q, q^{-1}]$-module ({\em resp.} sub-$\cit [\theta ,\theta^{-1}, q]$-module) of $G$ 
generated by  
$\varepsilon =(\varepsilon_{0}, \cdots , \varepsilon_{n-1})$ where $\varepsilon$ is defined by formula (\ref{eq:baseBir}),
\item $H_{0}$ ({\em resp.} $H^{log}_{0}$) the sub-$\cit [\theta , q, q^{-1}]$-module ({\em resp.} sub-$\cit [\theta , q]$-module) of $G$ generated by $(\varepsilon_{0}, \cdots , \varepsilon_{n-1})$,
\item $H_{2}$ ({\em resp.} $H^{log}_{2}$) the sub-$\cit [\theta , q, q^{-1}]$-module ({\em resp.} sub-$\cit [\theta , q]$-module) of $G$ generated by $(\varepsilon_{0}, \cdots , \varepsilon_{n-2})$.
\end{itemize}
We shall see that these modules are free. 
$H_{0}$ is the counterpart of the Brieskorn lattice $G_{0}$ in the tame case
and $H^{log}$ provides a canonical logarithmic extension of $H$ along $q=0$ (the eigenvalues of the residue matrix are all equal to $0$).\\

\begin{proposition}\label{prop:Birkhoff}
The matrix of $\nabla$ takes the form, in the system of generators $\varepsilon$ of $H_{0}^{log}$,
\begin{equation}\label{eq:Birkhoff}
(\frac{A_{0}(q)}{\theta}+A_{1})\frac{d\theta }{\theta}-(n-1)^{-1}\frac{A_{0}(q)}{\theta}\frac{dq}{q}  
\end{equation}
where
$$A_{0}(q) =(n-1)\left ( \begin{array}{cccccc}
0  & 0 & .  &  0   & 2q  & 0\\
1  & 0 & .  &  0   &  0  & 2q\\
0  & 1 & .  &  0   &  0  & 0\\
.  & . & .  &  .   &  .  & .\\ 
0  & 0 & .  &  1   &  0  & 0\\
0  & 0 & .  &  0   &  1  & 0
\end{array}
\right )$$

\noindent and $A_{1} =\diag (0,1,\cdots , n-1)$. 
\end{proposition}

\begin{proof} 
First, we have
\begin{itemize}
\item $[q\frac{\partial f}{\partial q}\omega_{0}]=[u_{1}\omega_{0}]$,
\item $[q\frac{\partial f}{\partial q}u_{1}\cdots u_{n-i}\omega_{0}]=[u_{1}\cdots u_{n-i+1}\omega_{0}]$ pour $i=3,\cdots ,n-1$,
\item $[q\frac{\partial f}{\partial q}u_{1}\cdots u_{n-2}\omega_{0}]=2q[\omega_{0}]+2[u_{n-1}\omega_{0}]$,
\item $[q\frac{\partial f}{\partial q}u_{n-1}\omega_{0}]=q[u_{1}\omega_{0}]$
\end{itemize}
\noindent and this follows respectively from the following formulas:
\begin{itemize}
\item $q\frac{\partial f}{\partial q}=u_{1}-u_{1}\frac{\partial f}{\partial u_1}+u_{n-1}\frac{\partial f}{\partial u_{n-1}}$,
\item $q\frac{\partial f}{\partial q}u_{1}\cdots u_{n-i}=u_{1}\cdots u_{n-i+1}-u_{1}\cdots u_{n-i}u_{n-2}\frac{\partial f}{\partial u_{n-2}}
+u_{1}\cdots u_{n-i}u_{n-1}\frac{\partial f}{\partial u_{n-1}}$ si $i=3,\cdots ,n-1$,
\item $q\frac{\partial f}{\partial q}u_{1}\cdots u_{n-2}=2q+2u_{n-1}-
2u_{1}\cdots u_{n-1}\frac{\partial f}{\partial u_{n-1}}$,
\item $q\frac{\partial f}{\partial q}u_{n-1}=qu_1 -
qu_{1}\frac{\partial f}{\partial u_{1}}-qu_{n-1}\frac{\partial f}{\partial u_{n-1}}$,
\end{itemize}
This gives the matrix of $\nabla_{q\partial_{q}}$ and the remaining assertion follows from lemma \ref{lemma:var}.
\end{proof}

\begin{proposition}\label{prop:Birkhoff(suite)}
The $\cit [\theta , q]$-module $H_{0}^{log}$ is free of rank $n$ and
$(\varepsilon_0 ,\cdots ,\varepsilon_{n-1})$
is a basis of it.
\end{proposition}
\begin{proof}
Observe the following:

\begin{itemize}
\item $H_{2}$ is a free $\cit [\theta ,q,q^{-1} ]$-module of rank $n-1$, with basis
$(\varepsilon_0 ,\cdots ,\varepsilon_{n-2})$: 
$\varepsilon_0 ,\cdots ,\varepsilon_{n-2}$ are linearly independent because their classes in $G_{0}/\theta G_{0}$ are so, see proposition \ref{prop:G0SingIsolees} 
and lemma \ref{lemma:AnneauJacLibre}.
It follows that $H_{2}^{log}$ is free of rank $n-1$.
\item $H$ is, by definition, of finite type  and moreover equipped with a connection by proposition \ref{prop:Birkhoff}: 
it is thus free over $\cit [\theta ,\theta^{-1}, q, q^{-1}]$, see \cite[proposition 1.2.1]{S4}. It follows that $H_{0}$ is free over $\cit [\theta , q, q^{-1}]$. Indeed, let $\alpha_{0},\cdots ,\alpha_{r}$ 
be a basis of
$H$: for $i\in\{0,\cdots ,r\}$ there exists $d_{i}\in\nit$ such that
$\theta^{-d_{0}}\alpha_{0},\cdots ,\theta^{-d_{r}}\alpha_{r}$ generate $H_{0}$ over $\cit [\theta , q, q^{-1}]$ and there are no non trivial relations between 
these sections on $\cit [\theta , q, q^{-1}]$. It follows that $H_{0}^{log}$ is free over $\cit [\theta , q]$.
\item $H_{2}^{log}$ is a free sub-module of the free module $H_{0}^{log}$: the rank of $H_{0}^{log}$ is therefore greater or equal than $n-1$.
The free module $H_{0}^{log}$ has $n$ generators: its rank is therefore less or equal than $n$.
It follows that the rank of $H_{0}^{log}$ is equal to $n-1$ or $n$.
\item  Assume for the moment that the rank of $H_{0}^{log}$ is equal to $n-1$:
 one would have a relation  
\begin{equation}\nonumber
a_{0}(\theta ,q) \varepsilon_0+\cdots +a_{n-1}(\theta ,q) \varepsilon_{n-1}=0
\end{equation}
where the $a_i (\theta , q)$'s are homogeneous polynomials in $(\theta ,q)$
(recall that $q$ is of degree $n-1$ and $\theta$ is of degree $1$,
see section \ref{sec:var}). One would have $a_{n-1}(0,q)=1$ because  $[\varepsilon_{n-1}]=q[\varepsilon_{0}]$ modulo $\theta$ by equation (\ref{eq:RelAlgJacobi}),
and thus $a_{n-1}(\theta ,q)=1$ by homogeneity. Because $\varepsilon_{n-1}$ is of degree $n-1$, one would have finally
$$\varepsilon_{n-1}=(a_{0}q+b_{0}\theta^{n-1})\varepsilon_{0}+a_{1}\theta^{n-2}\varepsilon_{1}+\cdots +a_{n-2}\theta\varepsilon_{n-2}$$
Apply $\theta\nabla_{q\partial_{q}}$ to this formula: using the computations of proposition
\ref{prop:Birkhoff}, one gets   
$$(a_{0}q\theta +2a_{n-2}q\theta +a_{n-2}(a_{0}q\theta +b_{0}\theta^{n}))\varepsilon_{0}$$
$$+(a_{0}q +b_{0}\theta^{n-1}+a_{n-2}a_{1}\theta^{n-1} )\varepsilon_{1}+(a_{1}\theta^{n-2}+a_{n-2}a_{2}\theta^{n-2} )\varepsilon_{2}$$
$$+\cdots +(a_{n-3}\theta^{2}+a_{n-2}a_{n-2}\theta^{2} )\varepsilon_{n-2}=2q\varepsilon_{1}$$
It follows that
\begin{itemize}
 \item $a_{n-2}b_{0}=0$
\item  $a_{0}+2a_{n-2}+a_{n-2}a_{0}=0$
\item  $a_{0}=2$
\item $b_{0}+a_{n-2}a_{1}=0$
\item $a_{i}+a_{n-2}a_{i+1}=0$ pour $i=1,\cdots , n-3$
\end{itemize}
The first three equalities give $a_{0}=2$, $a_{n-2}=-\frac{1}{2}$ and $b_{0}=0$. From the following ones we get
$a_{1}=\cdots =a_{n-3}=0$ and finally $a_{n-2}=0$: this is a contradiction. We conclude that the rank of $H_{0}^{log}$ is not equal to $n-1$ .
\end{itemize}
\noindent Thus, $H_{0}^{log}$ is free of rank $n$ and
$(\varepsilon_0 ,\cdots ,\varepsilon_{n-1})$ is a basis of it because it provides            
a system of $n$ generators.
\end{proof}

\subsubsection{Proof of theorem \ref{theo:Principal} }
\label{sec:ProofTheoPpl}
We keep the notations of section \ref{sec:BirkhoffQuadric}.

\begin{theorem}\label{theo:decompG}
 We have a direct sum decomposition
\begin{equation}
G=H\oplus H^{\circ}
\end{equation}
of free $\cit [\theta , \theta^{-1}, q, q^{-1}]$-modules where $H$ is free of rank $n$ and is equipped with a connection making it isomorphic to the 
differential system associated with the small quantum cohomology of quadrics in $\ppit^n$.
\end{theorem}
\begin{proof}
The module $G/H$ is of finite type and therefore free because it is equipped with a connection as it follows from proposition
\ref{prop:Birkhoff}. We thus have the direct sum decomposition
$$G=H\oplus r(G/H)$$
where $r$ is a section of the projection $p:G\rightarrow G/H$. This gives the expected decomposition with $H^{\circ} := r(G/H)$.
The assertion about quantum cohomology follows from example \ref{ex:ExQuadPn} and formula (\ref{eq:Birkhoff}) {\em via} the correspondence $\varepsilon_{i}\leftrightarrow b_{i}$ where $b$ denotes the hyperplane class and $b_{i}$ the $i$-fold cup-product of
$b$ by itself.
\end{proof}

\noindent We expect $H^{\circ} =0$, see corollary \ref{coro:RangG}. Notice that we do not assert 
in the theorem that $H^{\circ}$ is equipped with a connection. 

\begin{remark}
$H$ has only two slopes, $0$ and $1$. In particular, $H$ is the (localized) Fourier transform of a regular holonomic module $M$ whose singular points run through $C(f)\cup \{0\}$.
Moreover, 
$$\Rang H = \Irr (H)+\Reg (H)$$
where $\Irr (H)=n-1$ and $\Reg (H)=d-1$: 
indeed, $Q(\omega_0 )=0$ where
$$Q=\theta^n (\nabla_{\theta\partial_{\theta}})^{n}-2q (n-1)^{n-1}n \theta (\nabla_{\theta\partial_{\theta}})+2q(n-1)^n\theta$$
and $\omega_0$ is cyclic.
\end{remark}

\begin{remark}
 (Metric) In order to get a whole quantum differential system 
it remains to construct a flat ``metric'' on $H$, see f.i \cite{DoMa}.
If $S$ is a $\nabla$-flat, non degenerate bilinear  form on $H_{0}$, then
\begin{align}\nonumber
 \left\{ \begin{array}{l}
S(\varepsilon_{i}, \varepsilon_{j}) =S(\varepsilon_{0}, \varepsilon_{n-1})\in\cit^* \theta^{n-1}\ \mbox{si}\ i+j=n-1\\
S(\varepsilon_{i}, \varepsilon_{j}) =0\ \mbox{otherwise}
\end{array}
\right .
\end{align}
Conversely, all flat metrics are of this kind: as $A_{0}$ is cyclic, one can argue as in \cite{DoSa2}. 
\end{remark}

\section{Appendix: small quantum cohomology of hypersurfaces in projective spaces (overview)}

We briefly recall here the definition of the small quantum cohomology of smooth hypersurfaces in projectives spaces alluded to in this paper. Our references are \cite{CK}, \cite{GuestSakai}, \cite{MM} and \cite{Sak}.

\subsection{Small quantum cohomology}
Given a Fano projective manifold $M$ and a homology class $A\in H_{2}(M; \zit )$ 
one defines Gromov-Witten invariants (three points, genus $0$) $GW_{A}:H^{*}(M; \cit)^{3}\rightarrow\cit$ which satisfy the following properties:\\

\noindent {\bf Linearity.} $GW_{A}$ is linear in each variable.\\

\noindent {\bf Effectivity.} $GW_{A}$ is zero if $\int_{A}\omega_{M}<0$, $\omega_{M}$ denoting the symplectic form on $M$.\\

\noindent {\bf Degree.} Let $x$, $y$ and $z$ be homogeneous cohomology classes. Then $GW_{A}(x,y,z)=0$ 
if
$$\deg x+\deg y +\deg z\neq 2\dim_{\cit}M +2<c_{1}(M), A>$$  
$c_{1}(M)$ denoting the first Chern class of  $M$ and $<x,A>=\int_{A}x$.\\

\noindent {\bf Initialisation.} $GW_{0}(x,y,z)=\int_{M}x\cup y\cup z$.\\

\noindent {\bf Divisor axiom.} If $z$ is a degree $2$ cohomology class one  has
$GW_{A}(x,y,z)=<z ,A>GW_{A}(x,y,1 )$.\\

Assume that the rank of $H^{2}(M;\zit )$
is $1$ and let $p$ be a generator of it. 
Let $b_{0},\cdots , b_{s}$ be a basis of $H^{*}(M;\cit )$ and  $b^{0},\cdots , b^{s}$ its Poincar\'e dual. 
The small quantum product $\circ_{tp}$ (for short $\circ$) is defined by
\begin{equation}\label{eq:DefProdQuant}
x\circ_{tp}y =\sum_{i=0}^{s}\sum_{A\in H_{2}(M;\zit )}GW_{A}(x,y,b_{i})q^{A}b^{i}
\end{equation}
where $q^{A}=\exp (tA)$. It follows from the Fano condition that the sum (\ref{eq:DefProdQuant}) is finite, see for instance \cite[Proposition 8.1.3]{CK}.

\subsection{Small quantum cohomology of hypersurfaces in (weighted) projective spaces}

\subsubsection{Projective spaces}

\label{sec:PCQ}

Assume that $M=X_{d}^{n}$ is a degree $d\geq 1$ smooth hypersurface in $\ppit^{n}$ and let
$i:X_{d}^{n}\hookrightarrow \ppit^{n}$ be the inclusion. 
Let $p\in H^{2}(\ppit^{n} ;\cit )$ be the hyperplane class and $b=i^{*}p$.  Then $c_{1}(X_{d}^{n})=(n+1 -d)b$.  In what follows, we will assume that $n+1-d>0$ (Fano case). We have
\begin{align}\label{eq:CohXd}
 \left\{ \begin{array}{l}
H^{m}(X_{d}^{n} ; \cit )=H^{m}(\ppit^{n} ; \cit )\ \mbox{si}\ m<n-1 \\
H^{m}(X_{d}^{n} ; \cit )=H^{m+2}(\ppit^{n} ; \cit )\ \mbox{si}\ m>n-1
\end{array}
\right .
\end{align}
In particular, $H^{2}(X_{d}^{n} ; \cit )=H^{2}(\ppit^{n} ; \cit )$ if $n\geq 4$. The cohomology ring is divided in two parts:\\

\noindent {\bf The ambient part.} This is the space $H_{amb}(X_{d}^{n} ; \cit ):=\im i^{*}$, where $i^* : H^{*}(\ppit^{n} ;\cit )\rightarrow H^{*}(X^{n}_{d} ;\cit )$.
We have $H_{amb}(X_{d}^{n} ; \cit )=\oplus_{i=0}^{n-1}\cit b_{i}$ where $b_{i}=b\cup \cdots \cup b$ ($i$-times) and this is a cohomology algebra of rank $n$.\\

\noindent {\bf The primitive part.} This is $P(X_{d}^{n}):=\ker i_{!}\subset H^{n-1}(X^{n}_{d} ; \cit )$, 
where $i_{!}:H^{n-1}(X^{n}_{d} ;\cit )\rightarrow H^{n+1}(\ppit^{n} ;\cit )$ is the Gysin morphism.\\

\noindent  The small quantum cohomology of $X_{d}^{n}$ preserves the ambient part 
$H_{amb}(X_{d}^n ; \cit ) $, see \cite{Pan}, \cite[Chapter 11]{CK}. 
We thus get a subring denoted by $QH_{amb}(X_{d}^{n} ; \cit )$, equipped with the product $\circ$ 
and which describes the small quantum product of cohomology classes coming from the ambient space $\ppit^{n}$: 
using the degree property we get, for $0\leq m\leq n-1$,  
\begin{equation}\label{eq:FormulePCQ}
b\circ b_{n-1-m}=b_{n-m} +\sum_{\ell \geq 1}L_{m}^{\ell}q^{\ell }b_{n-m-\ell (n+1-d)}
\end{equation}
and
\begin{equation}\label{eq:FormulePCQ0}
b\circ b_{n-1}=\sum_{\ell\geq 1}L_{0}^{\ell}q^{\ell}b_{n-\ell (n+1-d)} 
\end{equation}
where $L_{m}^{\ell}\in\cit$ and $q^{\ell}=\exp (t\ell A)$,
$A$ denoting a generator of $H_{2}(X_{d}^{n};\zit )$; the constants $L_{m}^{\ell}$ vanish unless $0\leq m\leq n-(n+1-d)\ell$ and we have
$\deg q =(n+1-d)$, which is positive in the Fano case. {\em This is this product that we consider in these notes}.

Last, let us make the link between the small quantum cohomology and the quantum differential operators defined in section \ref{sec:ODQ}. The differential system associated with $X_{d}^{
n}$ is
\begin{align}\label{eq:SystDiffCohQuant}
\left\{ \begin{array}{l} 
\theta q\partial_{q} \varphi_{n-1-m}(q)=\varphi_{n-m}(q)+\sum_{\ell\geq 1}L_{m}^{\ell}q^{\ell}\varphi_{n-m-\ell (n+1-d)}(q) \ \mbox{pour}\ m=1,\cdots ,n-1\\ 
\\
\theta q\partial_{q} \varphi_{n-1}(q)=\sum_{\ell\geq 1}L_{0}^{\ell}q^{\ell}\varphi_{n-\ell (n+1-d)}(q) 
\end{array}
\right .
\end{align}
\noindent see formula (\ref{eq:FormulePCQ}) and (\ref{eq:FormulePCQ0}). 
It follows from \cite{Giv} that this system can be written 
\begin{equation}\nonumber
P(\theta q\partial_q , q, \theta )\varphi_0 (q)=[(\theta q\partial_q )^{n}-qd^{d}(\theta q\partial_q +\frac{1}{d}\theta )\cdots (\theta q\partial_q +\frac{d-1}{d}\theta )]\varphi_{0}(q)=0
\end{equation}
In other words, the matrix of system (\ref{eq:SystDiffCohQuant}) is conjugated to a companion matrix whose characteristic polynomial is $P(X, q, \theta )$. This allows to compute the constants $L_{m}^{\ell}$. 

\begin{remark}\label{rem:relationPCQ}
A first consequence is the formula
\begin{equation}\label{eq:RelPCQ}
b^{\circ n} =qd^d b^{\circ d-1}
\end{equation}
see for instance \cite[page 364]{CK}, 
which reads $P(\theta q\partial_q , q, 0 )=0$ via the correspondences
\begin{equation}\label{eq:correspondance}
b\circ \leftrightarrow \theta q\partial_q\ \mbox{and}\ 1\leftrightarrow \varphi_0
\end{equation}
A justification is the following: in the basis
$(\varphi_0 ,\theta q\partial_{q}\varphi_0 , \cdots , (\theta q\partial_{q})^{n-1}\varphi_0 )$
the matrix of $\theta q\partial_q$ is
$\Omega_0 +\theta [\cdots ]$
where $\Omega_0$ is a matrix with coefficients in $\cit [q]$ and whose characteristic polynomial is $P(\theta q\partial_q , q, 0 )$. 
Up to conjugacy, the matrix $\Omega_0$ is also the one of $\theta q\partial_{q}$
in the basis  $(\varphi_0 , \cdots ,  \varphi_{n-1})$.
\end{remark}

\begin{example}
\label{ex:ExQuadPn}
Let us consider the quadric in $\ppit^{n}$. 
In the basis $(\varphi_0 , \cdots ,  \varphi_{n-1})$, the matrix of $\theta q\partial_{q}$ takes the form 
$$\left ( \begin{array}{cccc}
0  & \cdots & 2q & 0\\
1  & \cdots & 0 & 2q\\
0  & \cdots & 0 & 0\\
0  & \cdots & 1 & 0
\end{array}
\right )$$ 
It is also the matrix of $b\circ$, using the correspondences (\ref{eq:correspondance}).
\end{example}

\subsubsection{Weighted projective spaces}
\label{sec:PCQPoids}

Let $M=X_{d}^{w}$ be a degree $d\geq 1$ hypersurface in $\ppit (w):=\ppit (w_{0},\cdots ,w_{n})$, satisfying the assumptions of theorem \ref{theo:hyperlisse}. 
Let $i:X_{d}^{w}\hookrightarrow \ppit (w)$ be the inclusion, $p\in H^{2}(\ppit (w) ;\cit )$ the hyperplane class and $b=i^{*}p$. 
By proposition \ref{prop:adjonction}, the first Chern class $c_{1}(X_{d}^{w})$ is $(w -d)b$ and we will assume in what follows that $w-d>0$ 
(Fano case, recall that $w=w_0 +\cdots +w_n$).
The cohomology $H^{m}(\ppit (w) ; \cit )$ groups of the untwisted sector are of rank $1$ if $m$ 
is even, they vanish otherwise and 
\begin{align}
 \left\{ \begin{array}{l}
H^{m}(X_{d}^{w} ; \cit )=H^{m}(\ppit (w) ; \cit )\ \mbox{si}\ m<n-1 \\
H^{m}(X_{d}^{w} ; \cit )=H^{m+2}(\ppit (w) ; \cit )\ \mbox{si}\ m>n-1
\end{array}
\right .
\end{align}
\noindent see \cite[Corollary 2.3.6 et 4.2.2]{Dolgachev} and \cite[Theorem 7.2]{Fletcher}.
As before, we divide the cohomology ring $H^{*}(M;\cit )$ into an ambient part $H_{amb}(X_{d}^{w} ; \cit ):=\im i^{*}$, where
$i^* : H^{*}(\ppit (w) ;\cit )\rightarrow H^{*}(X^{w}_{d} ;\cit )$ and a primitive part.
We thus have $H_{amb}(X_{d}^{w} ; \cit )=\oplus_{i=0}^{n-1}\cit b_{i}$ where $b_{i}=b\cup \cdots \cup b$ ($i$-times). The small quantum product of $X_{d}^{w}$ 
should preserves this ambient part and one would at the end get a subring $QH_{amb}(X_{d}^{w} ; \cit )$, equipped with a product $\circ$.
The differential system associated with this small quantum product looks like
(compare with (\ref{eq:SystDiffCohQuant}))
\begin{align}\label{eq:SystDiffCohQuantPoids}
\left\{ \begin{array}{l} 
\theta q\partial_{q} \varphi_{n-1-m}(q)=\varphi_{n-m}(q)+\sum_{\ell\geq 1}L_{m}^{\ell}q^{\ell}\varphi_{n-m-\ell (w -d)}(q) \ \mbox{pour}\ m=1,\cdots ,n-1\\
\\
\theta q\partial_{q} \varphi_{n-1}(q)=\sum_{\ell\geq 1}L_{0}^{\ell}q^{\ell}\varphi_{n-\ell (w -d)}(q) 
\end{array}
\right .
\end{align}
\noindent where $q$ is now of degree $w-d>0$. 
Following \cite{Giv}, \cite{HV} and \cite[section 5]{GuestSakai} this systems should be equivalent to the equation $P_H (\varphi_{0}(q))=0$ where $P_H$ is the differential operator defined by formula (\ref{eq:OpDiffQuant}).
Again, one can derive from this the constants $L_{m}^{\ell}$ in terms of combinatorial data.
A consequence is the formula 
\begin{equation}\label{eq:RelCohQuantPoidsEx}
b^{\circ n} =q\frac{d^d}{\prod_{i=1}^{n}w_{i}^{w_{i}}} b^{\circ d+n-w}
\end{equation}
as in remark \ref{rem:relationPCQ}.

\end{document}